\documentclass[oneside]{amsart}
\usepackage[utf8]{inputenc}

\usepackage{euscript}
\usepackage[T2A]{fontenc}
\usepackage[utf8]{inputenc}

\usepackage[english]{babel} 
\usepackage{amsmath}
\usepackage{amsfonts}
\usepackage{amssymb, amsthm}

\usepackage{hyperref}

\usepackage{amsmath}
\usepackage{graphicx}
\usepackage{geometry}

\usepackage{color}

\usepackage{multicol} 

\usepackage{graphicx}
\graphicspath{}
\DeclareGraphicsExtensions{.pdf,.png,.jpg} 

\newcommand{\mbbN}{\mathbb{N}}

\newcommand{\mbbH}{\mathbb{H}}

\newcommand{\eps}{\varepsilon}
\newcommand{\del}{\delta}

\newcommand{\intl}{\int \limits}
\newcommand{\suml}{\sum \limits}

\newcommand{\infl}{\inf \limits}
\newcommand{\supl}{\sup \limits}

\newcommand{\abs}[1]{{\left| #1 \right|}}

\newcommand{\wpart}[1]{{\left[ #1 \right]}}

\DeclareMathOperator{\diam}{diam}

\renewcommand{\le}{\leqslant}
\renewcommand{\ge}{\geqslant}

\theoremstyle{plain}
\newtheorem{thm}{Theorem} 
\newtheorem*{thm*}{Theorem}
\newtheorem*{lm*}{Lemma}
\newtheorem{lm}{Lemma}
\newtheorem{prop}{Proposition}

\theoremstyle{definition}
\newtheorem{defn}{Definition}
\newtheorem{cor}{Corollary}

\theoremstyle{remark}

\newcommand{\lr}[1]{{\left( #1\right)}}
\newcommand{\systemU}{{U_{\mathcal{A}}}}

\newcommand{\systemUt}{{\tilde{U}_{\mathcal{A}}}}

\begin{document}

\title{Scaling Entropy of Unstable Systems}
\author{Georgii Veprev}
\thanks{The work is supported by Ministry of Science and Higher Education of the Russian Federation, agreement №~075-15-2019-1619. The work is also supported by the V. A. Rokhlin scholarship for young mathematicians.}
\date{\today}
\address{Leonhard Euler International Mathematical Institute in St. Petersburg, \newline 14th Line 29B, Vasilyevsky Island, St. Petersburg, 199178, Russia}
\email{egor.veprev@mail.ru}

\maketitle

\begin{abstract}
In this work, we study the slow entropy type invariant of a dynamical system proposed by  A.~M.~Vershik. We provide an explicit construction of a system whose \emph{class of scaling entropy sequences} is empty. For this \emph{unstable} case, we introduce an upgraded notion of the invariant, generalize subadditivity results, and provide a complete series of examples.      
\end{abstract}

\section{Introduction}
The classical notion of Kolmogorov--Sinai entropy is based on the dynamics of measurable partitions of a measure space. For the case of zero entropy systems, A.~M.~Vershik proposed (see~\cite{V1, V3}) a new approach based on the dynamics of functions of several variables.
In topological dynamics, a metric space is usually fixed, and invariant measures can be considered on it. Oppositely to that, we will implement Vershik's approach which is the following. 
We assume that an automorphism of a measure space is given and vary a measurable metric (or semimetric) on the space.
For some metric and sufficiently small $\eps > 0$ we consider the sequence of \emph{$\eps$--entropies} of averaged metrics. This family of sequences increases in $\eps$ and often has a limit when $\eps$ goes to zero.  If this limit exists, we call the system \emph{stable}. The limit itself does not depend on the choice of the metric (see~\cite{Z1}) and is proved to be subadditive (see~\cite{PZ}). Note that similar constructions (\emph{a measure--theoretical complexity}) were considered in~\cite{F} and~\cite{KT}. 

In this work, we show that this limit does not necessarily exist in the classical sense. However, the invariant can be extended to the general case (see section~\ref{sec_invariance}). We also prove its subadditivity and construct a complete family of examples. Now let us proceed to the formal definitions.  


Let $(X, \mu, \rho)$ be a {standard probability} measure space endowed with a measurable semimetric $\rho$ on $X$, meaning that $\rho(x,y)$ is a symmetric non--negative measurable function on $(X^2, \mu^2)$ satisfying the triangle inequality. 
For a given positive $\eps$ define an \textit{$\eps$--entropy} of $(X, \mu, \rho)$ in the following way.
\begin{defn}
	Let $k$ be the minimal positive integer such that $X$ decomposes into a union of measurable sets $X_0, X_1, \ldots, X_k$ with $\mu(X_0) < \eps$ and $\diam_\rho(X_i)  < \eps$ for all $i>0$. Put
		\[
		\mathbb{H}_\eps(X, \mu, \rho) = \log_2 k.
		\]
	If there is no such $k$, put $\mathbb{H}_\eps(X, \mu, \rho) = +\infty$.
\end{defn}

Assume that for some semimetric $\rho$ its $\eps$--entropies are finite for all positive $\eps$. In~\cite{VPZ} it is shown that this property is equivalent to the separability of $\rho$ on the set of full measure. In this case, the semimetric is called \emph{admissible}. 

Now let $T$ be an  invertible measure--preserving transformation of the standard measure space $(X,\mu)$. {For $n\in \mathbb{N}$} denote by $T_{av}^n\rho$ the averaged by $T$ semimetric:
	\[
	T_{av}^n\rho(x, y) = \frac{1}{n} \sum\limits_{k = 0}^{n-1} \rho(T^k x, T^k y), {\quad x,y \in X}.
	\] 
Clearly, if $\rho$ is admissible, then $T_{av}^n\rho$ is admissible too.

Another condition requires the semimetric to be non-trivial.
We call $\rho$ \emph{generating} if there exists a set $X^\prime \subset X$ of full measure such that for any $x, y \in X^\prime$ there is $n \in \mbbN$ with $T_{av}^n\rho(x, y) > 0$. For example, any measurable metric is generating. 

Consider the following function 
$$
\Phi_\rho(n, \eps) = \mbbH_\eps(X, \mu, T_{av}^n\rho), {\quad n \in \mathbb{N}, \eps>0}.
$$  
Note that $\Phi_\rho(n,\eps) < +\infty$ for all $\eps$ and $n$ {provided} $\rho$ is admissible. Generally, the function $\Phi$ depends on $\eps$, $n$, and the semimetric $\rho$. However, its asymptotic behavior in some sense does not depend on $\eps$ and $\rho$, i.\,e. it forms the isomorphism invariant of the dynamical system. The following result is proved in~\cite{Z1}.    
\begin{thm*}[Zatitskiy, 2015]
	Let $T$ be an automorphism of the standard measure space $(X, \mu)$.  Suppose that for some admissible generating semimetric $\rho$ on $(X,\mu)$ there is a sequence $\{h_n\}$ such that for all sufficiently small $\eps > 0$
	\[
	\Phi_\rho(n, \eps) \asymp h_n.
	\]  
	Then for any admissible {generating semimetric} $\omega$ and any $\eps$ small enough
	\[
	\Phi_\omega(n, \eps) \asymp h_n.
	\] 
\end{thm*}
{Here and in what follows, for two sequences $\phi(n)$ and $\psi(n)$, relation $\phi \asymp \psi$ means that there exist two positive constants $c$ and $C$ with $c\phi(n) \le \psi(n) \le C \phi(n)$ for all $n \in \mbbN$.} 
\begin{defn}
	In this case, we call the sequence $h_n$ \emph{a scaling entropy sequence} of $(X, \mu, T)$. The system itself we call \emph{stable}.	
\end{defn}
    
Some important properties of the system can be described in terms of its scaling sequence. For example, if Kolmogorov--Sinai entropy is positive, then one could choose $h_n = n$ as a scaling entropy sequence of the system. In~\cite{VPZ} it is shown that the scaling sequence is bounded if and only if the automorphism has a pure point spectrum. In~\cite{PZ} F.~V.~Petrov and P.~B.~Zatitskiy proved that {if a scaling sequence exists, it can be chosen to be} increasing and subadditive. Conversely, there exists a stable ergodic system with a given increasing subadditive scaling entropy sequence, i.\,e., the complete classification of possible scaling sequences was obtained in the \emph{stable case}.    
    
Until now it was unclear whether or not unstable systems exist. In this work, we give a positive answer to this question and generalize the scaling entropy sequence invariant for the general case. Also, we construct an exhaustive family of examples.

\subsection*{Acknowledgments}
The author would like to thank Anatoly Vershik for his attention
to this work. The author is also grateful to his supervisor Pavel Zatitskiy for many helpful discussions.

\section{Construction of an unstable system}\label{secEx}
\begin{thm}
	There exists an ergodic system $(X, \mu, T)$ and an admissible semimetric $\rho$ on $X$ such that the asymptotic behavior of $\mbbH_\eps(X, \mu, T_{av}^n \rho)$ essentially depends on~$\eps$, meaning that for any $\eps > 0$ there exists $\delta$, $\eps > \del > 0$, with 
	\[
	\mbbH_{\eps}(X, \mu, T_{av}^n \rho) \lnsim 
	\mbbH_{\del}(X, \mu, T_{av}^n \rho).
	\]   
\end{thm}
Here, $\phi \lesssim \psi$ means that there is a positive constant $C$ such that $\phi(n) < C\psi(n)$ for all $n \in \mbbN$. We write $\phi \lnsim \psi$ if $\phi \lesssim \psi$ but not $\phi \asymp \psi$.
\begin{proof}
    Let $\mathcal{A} = \{\phi_m\}_{m=1}^\infty$ be an infinite family of subadditive functions {on $\mathbb{N}$} such that $\phi_m \lnsim \phi_{m+1}$. One could choose $\phi_m(n) = \log ^ m (n)$, for instance. By~\cite{Z2}, there is a family of corresponding ergodic systems $S_m = (X_m, \mu_m, T_m)$ such that for all $m$, any admissible generating semimetric $\rho_m$ on $(X_m, \mu_m)$, and all positive $\eps$ small enough
    \[
    \mbbH_\eps(X_m, \mu_m, T_{av}^n\rho_{{m}}) \asymp \phi_m(n).
    \]
    {For each $m \in \mbbN$ fix an admissible generating semimetric $\rho_m \le 1$ on $(X_m, \mu_m)$. Then define a system~$\systemUt$ as a product of  $S_m$:}
    \[
    \systemUt = \left(
    \prod_{m=1}^\infty X_m, 
    \prod_{m=1}^\infty \mu_m, 
    \prod_{m=1}^\infty T_m \right),
    \]
    where the automorhism $\prod_{m=1}^\infty T_m$ acts independently on each factor.  Now, by the ergodic decomposition theorem, there exists an ergodic measure $\mu$ on $\prod_{m=1}^\infty X_m$ such that all its coordinate projections coincide with the initial measures $\mu_m$. Let us change the measure to obtain an ergodic system 
    \[
    \systemU = \left(
    \prod_{m=1}^\infty X_m, 
    \mu, 
    \prod_{m=1}^\infty T_m \right).
    \]
    Define a semimetric $\rho$ on the product space such that for all $x, y \in \prod_{m=1}^\infty X_m$
    \[
    \rho(x,y) = \suml_{m=1}^{\infty} \frac{1}{2^m}\rho_m(x_m,y_m).
    \]
    Clearly, $\rho$ is generating and as we will show below, all its $\eps$--entropies are finite for all $\eps$. Therefore, $\rho$ is admissible and generating. 
    \begin{lm}\label{lmEx}
        Let $\systemU$ and $\rho$ be the product system with the semimetric described above. Then 
        \[
         \mbbH_\eps(\systemU, T_{av}^n\rho) \le \suml_{m=1}^{R(\eps)} \mbbH_{\frac{\eps}{2R(\eps)}}(X_m, \mu_m, \left(T_m\right)_{av}^n\rho_m), \quad {\eps>0}, 
        \]
        where $R(\eps) = -\log(\eps)$.
    \end{lm}
    \begin{proof}
        Let us fix some $n \in \mbbN$ and $\eps > 0$. For each $X_m$ consider its decomposition into subsets $A^{(m)}_0, \ldots, A^{(m)}_{k_m}$ such that $\mu_m\left(A^{(m)}_0\right) < \frac{\eps}{2R(\eps)}$ and $\diam_{T^n_{av}\rho_m}\left(A^{(m)}_i\right) < \frac{\eps}{2R(\eps)}$ for $i > 0$. Here 
        $$\log k_m =\mbbH_{\frac{\eps}{2R(\eps)}}(X_m, \mu_m, \left(T_m\right)_{av}^n\rho_m).$$ 
        Let $\pi_m$ be a standard projector onto $X_m$. Construction of $\mu$ implies that $\pi_m$ is measure--preserving. Denote 
        \[
        \hat A^{(m)}_i = \pi_m^{-1}\left( A^{(m)}_i\right).  
        \]
        Define a new error set 
        $
        K_0 = \bigcup\limits_{m = 1}^{R} \hat A^{(m)}_0,
        $
         { where $R = R(\eps) = -\log(\eps)$.}
        Clearly,
        \[
        \mu\lr{K_0} \le \suml_{m = 1}^{R} \hat A^{(m)}_0 \le \frac{\eps}{2}.
        \]
        For every $J = (j_1, \ldots, j_{R})$, where $j_m$ lies in $\{1, \ldots, k_m\}$, define 
        \[
        K_J = \bigcap\limits_{m = 1}^{R} \hat A^{(m)}_{j_m} \setminus K_0.
        \]
        Note that
        \[
        T^n_{av}\rho(x,y) = \suml_{m=1}^{\infty} \frac{1}{2^m}T^n_{av}\rho_m(x,y).
        \]
        Therefore, 
        \begin{multline*}
            \diam_{T^n_{av}\rho}\lr{K_J} \le 
            \suml_{m=1}^{\infty} \frac{1}{2^m} \diam_{T^n_{av}\rho_m}\pi_m\lr{K_J} \le \\
            \suml_{m=1}^{R} \frac{1}{2^m} \diam_{T^n_{av}\rho_m}A_{j_m}^{(m)}  + 
            \suml_{m=R+1}^{\infty} \frac{1}{2^m} \diam_{T^n_{av}\rho_m}X_m  \le \\
            R \cdot \frac{\eps}{2R} + 2^{-R -1} < 
            \frac{\eps}{2} + \frac{\eps}{2} = \eps.
        \end{multline*}
        So, we have constructed a partition $\mathcal{K} = \{K_J\}_{J} \cup \{K_0\}$ such that $\mu\lr{K_0} < \eps$ and $\diam_{T^n_{av}\rho}K_J < \eps$. The cardinality of  $\mathcal{K}$ does not exceed $k_1 \cdot k_2\cdot\ldots\cdot k_R + 1$. Thus,
        \[
        \mbbH_\eps(\systemU, T_{av}^n\rho) \le \log\lr{\abs{\mathcal{K}} - 1} \le 
        \suml_{m=1}^{R(\eps)} k_m = 
        \suml_{m=1}^{R(\eps)} \mbbH_{\frac{\eps}{2R(\eps)}}(X_m, \mu_m, \left(T_m\right)_{av}^n\rho_m),
        \]
        as desired. 
    \end{proof}
    
    Now, assume that $\systemU$ is stable, i.\,e., there exists some $\eps_0 > 0$ such that for any $\eps \le \eps_0$ the following equivalence holds:
    \[
    \mbbH_{\eps}\lr{X, \mu, T_{av}^n \rho} \asymp \mbbH_{\eps_0}\lr{X, \mu, T_{av}^n \rho}.
    \]
    Since, by Lemma~\ref{lmEx} one has 
    \begin{equation}\label{thEx_eq}
        \mbbH_{\eps}\lr{X, \mu, T_{av}^n \rho} \lesssim \suml_{k=1}^{R(\eps_0)} \mbbH_{\frac{\eps_0}{2R(\eps_0)}}(X_k, \mu_k, \left(T_k\right)_{av}^n) \lesssim \phi_{R(\eps_0)}\lr{n}.
    \end{equation}
    Let $h_n$ be a scaling entropy sequence of the system. By the previous formula, $h_n$ does not exceed $\phi_{R(\eps_0)}(n)$ asymptotically. However, in~\cite{Z2} it is proved that the scaling entropy sequence of a system grows not slower than the entropy sequence of a factor system. This implies that for any $m$
    \[
    h_n \gtrsim \phi_m(n).
    \]
    One could, for instance, choose $m = R(\eps_0) + 1$ and obtain a contradiction. Therefore, our assumption is false, and the system $\systemU$ is not stable.  
\end{proof}

\section{Invariance}\label{sec_invariance}
The purpose of this section is to generalize the notion of the  scaling entropy sequence invariant for the general (unstable) case. In this instance, it will be some equivalence class of functions of two variables. Let us define the equivalence relation.
\begin{defn}\label{defRel}
    Let {$\Phi, \Psi \colon \mathbb{N} \times \mathbb{R}_+ \to \mathbb{R}_+$} be two functions that decrease {with respect to their second arguments}. We will write that $\Phi \preceq \Psi$ if for any $\eps > 0$ there is some $\del > 0$ such that
    \[
    \Phi(n, \eps) \lesssim \Psi(n, \delta).
    \]
    We will call $\Phi$ and $\Psi$ equivalent if $\Phi \preceq \Psi$ and $\Psi \preceq \Phi$. {In this case, we will write $\Phi \sim \Psi$.}
\end{defn}
Clearly, relation $\preceq$ forms a partial order on the set of equivalence classes. We will denote by $\wpart{\Phi}$ the class of function $\Phi$.

Assume that a measure preserving system $\lr{X, \mu, T}$ with an admissible semimetric $\rho$ is given. 
Let us define a \emph{scaling entropy of the system $\lr{X, \mu, T}$ with respect to semimetric $\rho$} as the following equivalence class:  
\[
\mathcal{H}\lr{X, \mu, T, \rho} = \Big[\mbbH_\eps\lr{X, \mu, T_{av}^n \rho}\Big].
\]

In fact, this class does not depend on the generating admissible semimetric. Moreover, in~\cite{Z1} the following theorem is proved.
\begin{thm*}[Zatitskiy, 2015]
    Let $(X,\mu, T)$ be a measure preserving system. Assume that $\rho$ is an admissible generating semimetric with a finite integral over $(X^2, \mu^2)$. Then for any admissible semimetric $\omega$ with a finite integral and any $\eps > 0$ there exist some positive $c$ and $\del$ such that
    \[
    \mbbH_\eps\lr{X,\mu, T_{av}^n \omega} \le c\mbbH_\del\lr{X,\mu, T_{av}^n \rho}.
    \]
\end{thm*}
\begin{cor}[Invariance]
    Let $\rho$ and $\omega$ be two admissible generating semimetrics with finite integrals. Then
    \[
    \mathcal{H}\lr{X, \mu, T, \rho} = \mathcal{H}\lr{X, \mu, T, \omega}.
    \]
\end{cor}
This corollary allows us to give the following definition.
\begin{defn}
    The \emph{scaling entropy of the system} $\lr{X, \mu, T}$ is the following class
    \[
    \mathcal{H}\lr{X, \mu, T} = \mathcal{H}\lr{X, \mu, T, \rho},
    \]
    where $\rho$ is an arbitrary admissible generating semimetric with a finite integral. 
\end{defn}
An important example of such a semimetric is {a cut} semimetric corresponding to a (countable) generating partition.
Another useful corollary gives an estimate of the scaling entropy of a factor system.
\begin{cor}\label{remFactor}
    Let a system $\lr{\hat X, \hat \mu, \hat T}$ be a factor of a measure--preserving system $\lr{X, \mu, T}$. Then
    \[
    \mathcal{H}\lr{\hat X, \hat \mu, \hat T} \preceq \mathcal{H}\lr{X, \mu, T}.
    \]
\end{cor}

\section{Subadditivity}\label{secSub}
In this section, we prove that the scaling entropy  of a system is increasing and subadditive with respect to $n$. Conversely, all increasing {in $n$} subadditive functions (decreasing in $\eps$) can be obtained as a scaling entropy of some automorphisms.
\begin{defn}
    We call a function $\Phi(n,\eps)$ \emph{subadditive} if for all $\eps > 0$, and for any $k,m \in \mbbN$
    \[
    \Phi(k+m, \eps) \le \Phi(k,\eps) + \Phi(m, \eps).
    \]
\end{defn}

\begin{thm}\label{subthm}
    Let $\lr{X, \mu, T}$ be a measure preserving system. Then there exists an increasing in $n$ and decreasing in $\eps$ subadditive function $\Phi(n, \eps)$ such that
    \[
    \Phi \in \mathcal{H}\lr{X, \mu, T}.
    \]
\end{thm}

\begin{proof}
    We will use the following estimates proved in~\cite{PZ}. 
    \begin{lm}[Petrov, Zatitskiy, 2015]\label{lm_PZ}
        Let $\rho_1, \ldots, \rho_k$ be admissible semimetrics on the measure space~$\lr{X,\mu}$, and $\rho_i \le 1$ for all $i \le k$.
        \begin{enumerate}
            \item
            Suppose that {$\eps>0$ and} $\mbbH_\eps\lr{X, \mu, \rho_i} > 0$ {for all $i\le k$}. Then 
            \[
            \mbbH_{2\sqrt{\eps}}\lr{X,\mu, \frac{1}{k} \suml_{i = 1}^k \rho_i} \le
            2\suml_{i = 1}^k\mbbH_\eps(X,\mu, \rho_i).
            \]
            
            \item There exists some $m \le k$ such that
            \[
            \mbbH_{2\sqrt{\eps}}\lr{X,\mu, \rho_m} \le 
            \mbbH_{\eps}\lr{X,\mu, \frac{1}{k} \suml_{i = 1}^k \rho_i}.
            \]
        \end{enumerate} 
        
    \end{lm}
    
    Now, let $\rho \le 1$ be an admissible generating semimetric. Denote $\Psi(m, \eps) = \mbbH_\eps\lr{X, \mu, T_{av}^m \rho}$. Let us prove the following.
    \begin{prop}\label{prop1}
        The following inequalities hold.
        \begin{enumerate}
        \item For all $k, n \in \mbbN$ and positive $\eps$ small enough
            \[
            \Psi(kn,\eps) \le 2k \Psi\lr{n, \frac{\eps^2}{4}}.
            \]
        \item For all $k, n \in \mbbN$ with $k \le n$ one has 
            \[
            \Psi(n,\eps) \ge  \Psi\lr{k, 2\sqrt{2\eps}}.
            \]
            
         \end{enumerate}
    \end{prop}
    \begin{proof}
        Let us prove the first inequality. Note that for $\eps < \frac{1}{3} \intl_{X^2}\rho$ the $\eps$--entropy of $\rho$ is positive. Also, any averaged semimetric $T_{av}^m\rho$ has the same integral as the initial one, therefore, its $\eps$--entropy is positive too. Let us fix such $\eps > 0$ and some $k, n \in \mbbN$. Define for $i \le k$
        \[
        \rho_i = T^{(i-1)n}\ T_{av}^n \rho.
        \]
        Clearly, 
        \[
        T_{av}^{kn}\rho = \frac{1}{k}\suml_{i = 1}^k\rho_i.
        \]
        Applying lemma \ref{lm_PZ}, we obtain:
        \[
        \mbbH_\eps\lr{X, \mu, T_{av}^{kn}{\rho}} \le 2\suml_{i = 1}^k \mbbH_{\frac{\eps^2}{4}}\lr{X, \mu, \rho_i} =
        2k \mbbH_{\frac{\eps^2}{4}}\lr{X, \mu, T_{av}^n\rho}.
        \]
        The first part is proved.
        
        Let us proceed to the second inequality. Let $n = km + r$, where $r < k$. Note that $r \le \frac{n}{2}$. Then
        \[
        T_{av}^n \rho \ge \frac{1}{n} \suml_{i = 0}^{km} T^i\rho \ge 
        \frac{n-r}{n} \frac{1}{km}\suml_{i = 0}^{km} T^i\rho \ge 
        \frac{1}{2}T_{av}^{km}\rho.
        \]
        Thus,
        \[
        \mbbH_\eps\lr{X, \mu, T_{av}^n\rho} \ge \mbbH_\eps\lr{X, \mu, \frac{1}{2}T_{av}^{km}\rho} \ge 
        \mbbH_{2\eps}\lr{X, \mu, T_{av}^{km}\rho}.
        \]
        Now apply the second part of lemma \ref{lm_PZ} for semimetrics $\rho_i = T^{(i-1)n}\ T_{av}^n \rho$. We obtain
        \[
        \mbbH_\eps\lr{X, \mu, T_{av}^n\rho} \ge \mbbH_{2\eps}\lr{X, \mu, T_{av}^{km}\rho} \ge 
        \mbbH_{2\sqrt{2\eps}}\lr{X, \mu, T_{av}^{k}\rho}.
        \]
    \end{proof}
    
    \begin{lm}\label{sublm}
        Let $\eta(n), \phi(n),$ and $\psi(n),$ $n \in \mathbb{N},$ be sequences of non-negative real numbers. Suppose that  
        \begin{equation}
        \label{eq1}
        \phi(kn) \le k \psi(n) \quad \text{for } k, n \in \mbbN,
        \end{equation}
        and  
        \begin{equation}
        \label{eq2}
        \phi(n) \ge \eta(k) \quad \text{for } k \le n.
        \end{equation}
        Then, there exists an increasing subadditive function $\theta(n)$ such that
        \[
        \eta(n)\le\theta(n) \le 2\psi(n).
        \]
    \end{lm}
    \begin{proof}
        Let 
        \[
        \hat \phi(n)  = \infl_{m \ge n} \phi(m) \le \phi(n).
        \]
        It is clear that $\hat \phi$ is increasing and by~\eqref{eq2} for any $k \le n$ 
        \[
        \hat \phi(n) \ge \eta(k).
        \]
        Also, by~\eqref{eq1} for all $k, n \in \mbbN$
        \[
        \hat \phi(kn)  \le \phi(kn) \le k \psi(n).
        \]
        Define $\hat \theta$ as follows:
        \[
        \hat\theta(n) = \supl_{k > 0} \frac{\hat\phi(kn)}{k} \le \psi(n).
        \]
        Obviously, $\hat\theta(n) \ge \hat\phi(n)$. Note that $\hat \theta$ increases because $\hat\phi$ increases, and
        \[
        \hat\theta(kn) = \supl_{m > 0}\frac{\hat\phi(mkn)}{m} = k\supl_{m > 0}\frac{\hat\phi(mkn)}{mk} \le k \supl_{l>0}\frac{\hat\phi(ln)}{l} = k\hat\theta(n).
        \]
        Now, define $\theta$:
        \[
        \theta(n) = n \supl_{m \ge n} \frac{\hat\theta(m)}{m} \ge \hat\theta(n) \ge \hat\phi(n) \ge \eta(n).
        \]
        Firstly, $\theta$ is increasing. Indeed, 
        \[
        \theta(n) = \max\lr{\hat\theta(n), n\supl_{m \ge n+1} \frac{\hat\theta(m)}{m}} \le 
        \max \lr{\hat\theta(n+1), (n+1)\supl_{m \ge n+1}\frac{\hat\theta(m)}{m}} = \theta(n+1).
        \]
        Secondly, $\theta$ is subadditive:
        \[
        \theta(k + n) = (k + n)\supl_{m \ge k+n}\frac{\hat\theta(m)}{m} \le  
        k\supl_{m \ge k}\frac{\hat\theta(m)}{m} + n\supl_{m \ge n}\frac{\hat\theta(m)}{m} = \theta(k) + \theta(n).
        \]
        It only remains to show that $\theta(n) \le 2 \psi(n)$. 
        \[
        \theta(n) = n \supl_{m \ge n} \frac{\hat\theta(m)}{m} \le
        n \supl_{m \ge n} \frac{\hat\theta\lr{n\wpart{\frac{m}{n}}+ n}}{m} \le
        \supl_{m \ge n} \frac{n\wpart{\frac{m}{n}}+n}{m} \ \hat\theta(n) \le 2\psi(n).
        \]
    \end{proof}
    Now let us complete the proof of the theorem. {For a fixed $\eps>0$  we use  Lemma~\ref{sublm} with $\eta(n) = \Psi\lr{n, 2\sqrt{2\eps}},$ $\phi(n) = \Psi(n,\eps)$, and $\psi(n)= 2\Psi\lr{n, \frac{\eps^2}{4}}$, the conditions~\eqref{eq1} and~\eqref{eq2} of Lemma~\ref{sublm} are guaranteed by Proposition~\ref{prop1}. We obtain an} increasing subadditive function $\Theta(n, \eps)$ such that
    \begin{equation}\label{thm_sub_eq1}
        \Psi\lr{n, 2\sqrt{2\eps}}\le \Theta(n, \eps) \le 4\Psi\lr{n, \frac{\eps^2}{4}}.
    \end{equation}
    {It only remains to get a decrease in $\eps$. To do that let us find a sequence $\{\eps_k\}_{k=1}^\infty$ such that $\Theta(n, \eps_k) < 4\Theta(n, \eps_{k+1})$ for all $k, n \in \mbbN$ and $\eps_k$ tends to 0. It is always possible due to inequality \ref{thm_sub_eq1}. For $\eps > 0$ denote by $\gamma(\eps)$ the smallest $k$ such that $\eps_k < \eps$. Put then
    \[
    \Phi(n, \eps) = 4^k \Theta\lr{n, \eps_{\gamma(\eps)}}.
    \]
    It is easy to see that ${\Phi\sim}\Theta \sim \Psi$, therefore $\Phi \in \mathcal{H}\lr{X, \mu, T}$, as desired.}
\end{proof}

The next theorem gives the opposite result. It completes the description of the possible values of the invariant. 
\begin{thm}\label{thmSubEx}
    Let $\Phi(n,\eps)$ be an increasing in $n$ and decreasing in $\eps$ subadditive function of two variables. Then there exists a measure preserving system $(X, \mu, T)$ such that
    \[
        \Phi \in \mathcal{H}\lr{X, \mu, T}.
    \]
\end{thm}
\begin{proof}
    We will use the construction of the unstable system described in section~\ref{secEx}. 
    Let 
    \[
    \phi_m(n) = \Phi\lr{n, \frac{1}{m}}.
    \]
    Note that $\phi_m(\cdot)$ is increasing and subadditive. Denote $\mathcal{A} = \left\{ \phi_m\right\}_{m=1}^\infty$. Let us construct the system~$\systemU$ and show that $\Phi \in \mathcal{H}\lr{\systemU}$. Lemma~\ref{lmEx} gives the upper bound {for any $\eps>0$ fixed}:
    \[
        \mbbH_\eps(\systemU, T_{av}^n\rho) \le \suml_{k=1}^{R(\eps)} \mbbH_{\frac{\eps}{2R(\eps)}}(X_k, \mu_k, \left(T_k\right)_{av}^n\rho_k) \lesssim \phi_{R(\eps)}(n).
    \]
    Therefore,
    \[
    \mathcal{H}\lr{\systemU} \preceq \Phi.
    \]
    However, $\systemU$ has a stable factor $(A_m, \mu_m, T_m)$. Then for all $m \ge 1$
    \[
    \mathcal{H}\lr{\systemU}\succeq \phi_m(n).
    \]
    Thus,  $\mathcal{H}\lr{\systemU} \ni \Phi$, as desired. 
\end{proof}

\section{On the minimality of the relation}
It might seem that the equivalence relation we defined is too strong. The question is in improving~$\delta$ in Definition~\ref{defRel}. Is it possible to control $\delta$ by some function of $\eps$? Note that the proof of subadditivity (Theorem~\ref{subthm}) requires  $\delta = 4\eps^2$ only. If there exists a weaker equivalence relation, the invariant can separate more systems. However, the following theorem claims that our relation is sharp in this sense.

\begin{thm}\label{minthm}
    Let $f\colon (0,1) \to (0,1)$ be an increasing function. Then there exist a measure preserving system $\lr{X,\mu, T}$ and two admissible generating semimetrics $\rho$ and $\omega$ on $X$ such that for any $\eps_0 > 0$ there is an $0 < \eps < \eps_0$ satisfying
    \begin{equation}\label{mineq}
		\mbbH_{\eps}(X, \mu, T_{av}^n \omega) \gnsim \mbbH_{f(\eps)} (X, \mu, T_{av}^n \rho).
	\end{equation}
\end{thm}

\begin{proof}
    Let us show the following proposition first.
    \begin{prop}\label{minprop}
    Let $h\colon (0,1) \to (0,1)$ be some map. Then there exists $\kappa \in (0,1)$ such that there is an increasing sequence in $h^{-1}(\kappa,1)$.
    \end{prop}
    The proof of the proposition is clear. Indeed, any subset of the real line without an infinite increasing sequence is countable.
    
    Now let $\{\psi_\alpha(n) \},\ {\alpha\in(0,1)}$ be a family of increasing subadditive functions such that 
    $\psi_{\alpha_1}(n)\lnsim \psi_{\alpha_2}(n)$ for any ${\alpha_1 < \alpha_2}$. One, for example, can use $\psi_\alpha(n) = \log^{1+\alpha}(n)$. Consider the corresponding stable measure preserving systems $\lr{Y_{\alpha}, \nu_{\alpha}, R_{\alpha}}$ and admissible generating semimetrics $\tau_\alpha \le 1$ such that
    for any positive $\eps < h(\alpha)$
    \[
	\mbbH_{\eps}(Y_\alpha, \nu_\alpha, (R_\alpha)_{av}^n \tau_\alpha) \asymp \psi_\alpha(n).
	\]
    Via Proposition~\ref{minprop}, we obtain some $\kappa > 0$ and an increasing sequence $\{\alpha_m\}_{m = 1}^{\infty}$ such that $h(\alpha_m) > \kappa$ for all $m > 0$. Let us define $\phi_m = \psi_{\alpha_m}$ and let $\mathcal{A}$ be a family of $\phi_m$.
    Now let us construct $\systemU$ {as in section~\ref{secEx}} using $\lr{Y_{\alpha_m}, \nu_{\alpha_m}, R_{\alpha_m}}= \lr{X_m, \mu_m, T_m}$ as coordinate factors. 
    
    Define $\rho$ as a standard semimetric on $\systemU$:
    \[
    \rho(x,y) = \suml_{m=1}^{\infty} \frac{1}{2^m}\rho_m(x_m,y_m),
    \]
    where $\rho_m = \tau_{\alpha_m}$. We will look for $\omega$ in the similar linear form:
    \[
	    \omega(x, y) = 
		\suml_{m=1}^{\infty} C_m \rho_m(x_m, y_m),
	\]
	where $1 > C_m > 0$ and $\sum_{m=1}^{\infty} C_m$ is finite. Note that any semimetric of such a type is always generating and admissible.  
	
	For all $m, n \ge 1$
	\[
	\mbbH_{\eps}\lr{\systemU, T_{av}^n \omega} \ge \mbbH_{\eps}\lr{\systemU, C_m T_{av}^n \rho_m} = \mbbH_\eps\lr{X_m, \mu_m, C_m\lr{T_m}_{av}^n}.
	\]
	Using that $C_m \le 1$ we obtain:
	\[
	\mbbH_{\eps}(X_m, C_m T_{av}^n\rho_m) \ge \mbbH_{\frac{\eps}{C_m}}(X_m, T_{av}^n\rho_m) \asymp \phi_m(n),
	\]
	while $\frac{\eps}{C_m} \le \kappa$.  Therefore,
	\[
	\mbbH_{\eps}(\systemU, T_{av}^n\omega) \gtrsim \phi_m(n),
	\]
	for all $m$ such that $C_m \ge \kappa^{-1}\eps$. Let $K(\eps)$ be the largest such $m$.
	
	However, Lemma~\ref{lmEx} gives an upper estimate of the $\eps$--entropy of $\rho$
	\[
	\mbbH_{\eps}(\systemU, T_{av}^n\rho) \lesssim \phi_{R(\eps)}(n),
	\]
	where $R(\eps) = -\log \eps$. Hence,
	\[
	\mbbH_{f(\eps)}(\systemU, T_{av}^n\rho) \lesssim \phi_{R(f(\eps))}(n).
	\]
	Inequality~\ref{mineq} holds for every $\eps$ such that
	\[
	K(\eps) > R(f(\eps)). 
	\]
	And this is true when
	\[
	C_{R(f(\eps))+1} > \kappa^{-1}\eps.  
	\]
	Now let $\eps_p = \kappa 2^{-p}$ for $p \in \mbbN$ and put $C_{R(f(\eps_p))+1} = 2\kappa^{-1}\eps_p \le 1$. If some $C_m$ is not defined yet put $C_m = 2^{-m}$. Clearly, obtained semimetric $\omega$ is desired. 
\end{proof}

\end{document}